\documentclass[psamsfonts]{amsart}

\usepackage{setspace}
\usepackage{tipa}
\usepackage{mathrsfs}
\usepackage{amsfonts}
\usepackage{amssymb,amsfonts}
\usepackage[all,arc]{xy}
\usepackage{enumerate}
\usepackage{mathrsfs}
\usepackage{yhmath}
\usepackage{float}
\usepackage{graphicx}
\def\dis{\displaystyle}

\newtheorem{thm}{Theorem}[section]
\newtheorem{cor}[thm]{Corollary}
\newtheorem{prop}[thm]{Proposition}
\newtheorem{lem}[thm]{Lemma}

\theoremstyle{definition}
\newtheorem{defn}[thm]{Definition}

\theoremstyle{remark}
\newtheorem{rem}[thm]{Remark}

\numberwithin{equation}{section}

\begin{document}

%
\title{On obstacle problem for mean curvature flow with driving force}

\thanks{The work of the first author is partly supported by the Japan Society for the Promotion of Science (JSPS) through the grants No.\,26220702 (Kiban S), No.\,17H01091 (Kiban A), No.\,16H03948 (Kiban B) and No.\,18H05323 (Kaitaku).
The work of the second author is partly supported by NSF grant DMS-1615944.
The third author is the Research Fellow of Japan Society for the Promotion of Science, Number: 17J05160.
Corresponding author: Longjie Zhang.}

\author[Y. Giga]{Yoshikazu Giga}
\address[Y. Giga]{Graduate School of Mathematical Sciences, 
University of Tokyo 3-8-1 Komaba, Meguro-ku, Tokyo, 153-8914, Japan}
\email{labgiga@ms.u-tokyo.ac.jp}

\author[H. V. Tran]{Hung V. Tran}
\address[H. V. Tran]
{Department of Mathematics, 
University of Wisconsin Madison, Van Vleck Hall, 480 Lincoln Drive, Madison, Wisconsin 53706, USA}
\email{hung@math.wisc.edu}

\author[L. Zhang]{Longjie Zhang}
\address[L. Zhang]{Graduate School of Mathematical Sciences, 
University of Tokyo 3-8-1 Komaba, Meguro-ku, Tokyo, 153-8914, Japan}
\email{zhanglj@ms.u-tokyo.ac.jp}

\keywords{mean curvature flow; driving force; obstacle problem; large time behavior; limiting profiles; radially symmetric setting.}

\subjclass[2010]{35A01, 35A02, 35K55, 53C44.}

\begin{abstract}
In this paper, we study an obstacle problem associated with the mean curvature flow with constant driving force. Our first main result concerns interior and boundary regularity of the solution. We then study in details the large time behavior of the solution and obtain the convergence result. In particular, we give full characterization of the limiting profiles in the radially symmetric setting.
\end{abstract}

\maketitle

%
%
%
%
%

\section{Introduction}\large

 In this paper, we study an obstacle problem for level-set forced mean curvature flow equation. We assume further that the surface evolution is described by the mean curvature with constant driving force $A$. Under the assumption, the equation is
\begin{equation}\label{eq:levelset}
u_t=|D u|\text{div}\left(\frac{D u}{|D u|}\right)+A|D u|,\quad \text{in}\ \mathbb{R}^N\times(0,T),
\end{equation}
\begin{equation}\label{eq:initial}
u(x,0)=u_0(x),\quad \text{on}\ \mathbb{R}^N,
\end{equation}
\begin{equation}\label{eq:obs}
\psi^-(x)\leq u(x,t)\leq \psi^+(x),\quad \text{on}\ \mathbb{R}^N\times[0,T).
\end{equation}
Here $T>0$, $\text{supp}\,\psi^{\pm}\subset \overline{\Omega}$ for some open, bounded, smooth domain $\Omega\subset \mathbb{R}^N$. The given positive constant $A$ is called driving force. Moreover we assume $u_0\in C^{1,1}(\mathbb{R}^N)$, and $\psi^-\leq u_0\leq \psi^+$ on $\mathbb{R}^N$, where $u_0$, $\psi^{\pm}$ are all $L$-Lipschitz continuous for some fixed $L>0$.
In this paper, by $C^{1,1}(\mathbb{R}^N)$ we mean the space of all functions on $\mathbb{R}^N$ whose first derivatives are Lipschitz continuous.
For $U\subset \mathbb{R}^N$, we say that a function $a:U \rightarrow \mathbb{R}$ is $L$-Lipschitz continuous if 
$$
|a(x)-a(y)|\leq L|x-y| \quad \text{for all $x,y\in U$.}
$$ 
We postpone the definition of viscosity solutions of (\ref{eq:levelset}), (\ref{eq:initial}), and (\ref{eq:obs}) until Section 2.

\subsection{Main results}

Here we give our main results.

\begin{thm}\label{thm:grad}
Let $u$ be the viscosity solution of {\rm (\ref{eq:levelset}), (\ref{eq:initial})}, and {\rm (\ref{eq:obs})}. Then 
$$
|u(x,t)-u(y,t)|\leq L|x-y|,
$$
and
$$
|u(x,t)-u(x,s)|\leq M|t-s|,
$$
for all $x,y\in\mathbb{R}^N$, $t,s\in(0,T)$. Here $M$ is a constant such that
$$
M>2\left\Vert D^2 u_0 \right\Vert_{L^{\infty}(\mathbb{R}^N)}+A L.
$$
\end{thm}

As we mention in Section 2 more precisely, the existence of a unique viscosity solution is known by \cite{M}. However, such a global estimate is new although the proof for Lipschitz bound is an adjustment of the proof without obstacles; see \cite{G} for the spatial bound.

\begin{thm}\label{thm:asym}
Let $u$ be the viscosity solution of {\rm (\ref{eq:levelset}), (\ref{eq:initial})}, and {\rm (\ref{eq:obs})}. There exists a $L$-Lipschitz function $v$ in $\mathbb{R}^N$ such that 
$$
u(\cdot,t)\rightarrow v,\quad as\ t\rightarrow \infty,
$$
uniformly in $\mathbb{R}^N$. Here $v$ is a stationary solution of {\rm(\ref{eq:levelset})}, and {\rm(\ref{eq:obs})}.
\end{thm}

A similar result is proved for the Neumann problem in a convex domain for the level-set mean curvature flow equation \cite{GOS} without obstacles. We shall adjust their proof for our setting.

We aim at characterizing the limiting profile $v$ in term of given initial condition $u_0$, and obstacles $\psi^\pm$. This is  a challenging task, and at this moment, we are able to get full characterization in the radially symmetric case. This is done by a careful study of radial solution of (\ref{eq:levelset})--(\ref{eq:obs}). A key observation is that (\ref{eq:levelset}) becomes a first order equation with singularity at the center of radial symmetricity. Here is our statement.

\begin{thm}\label{thm:asymcircle}
Let $\Omega=B_R(O)$, and 
$$
\psi^{+}(x)=\left\{
\begin{array}{lcl}
 \lambda {\rm dist}(x,\partial \Omega)=\lambda(R-|x|),\quad x\in\Omega,\\
0,\quad x\in \mathbb{R}^N\setminus \Omega,
\end{array}
\right.
$$
where $\lambda >0$ is given.

Assume $\psi^-<0$ in $\Omega$, and $\psi^-(x)$ is radial, that is,  it depends only on $|x|$. Moreover, assume that $u_0$ is radial, and $\psi^-\leq u_0\leq \psi^+$. The following holds.
\\
{\rm (1)} If $R\leq (N-1)/A$, then $u(\cdot,t)\rightarrow 0$ uniformly in $\mathbb{R}^N$, as $t\rightarrow\infty$.
\\
{\rm (2)} If $R>(N-1)/A$, then for $B=\max\limits_{(N-1)/A\leq |x|\leq R}u_0\geq0$,
\[
\text{$u(\cdot,t)\rightarrow \psi_B$ uniformly in $\mathbb{R}^N$, as $t\rightarrow\infty$, where, for $C\in\mathbb{R}$,}
\]
$$
\psi_C:=\psi^+\wedge C=\min\{\psi^+,C\}.
$$
\end{thm}
\begin{figure}[htbp]
	\begin{center}
            \includegraphics[height=8.0cm]{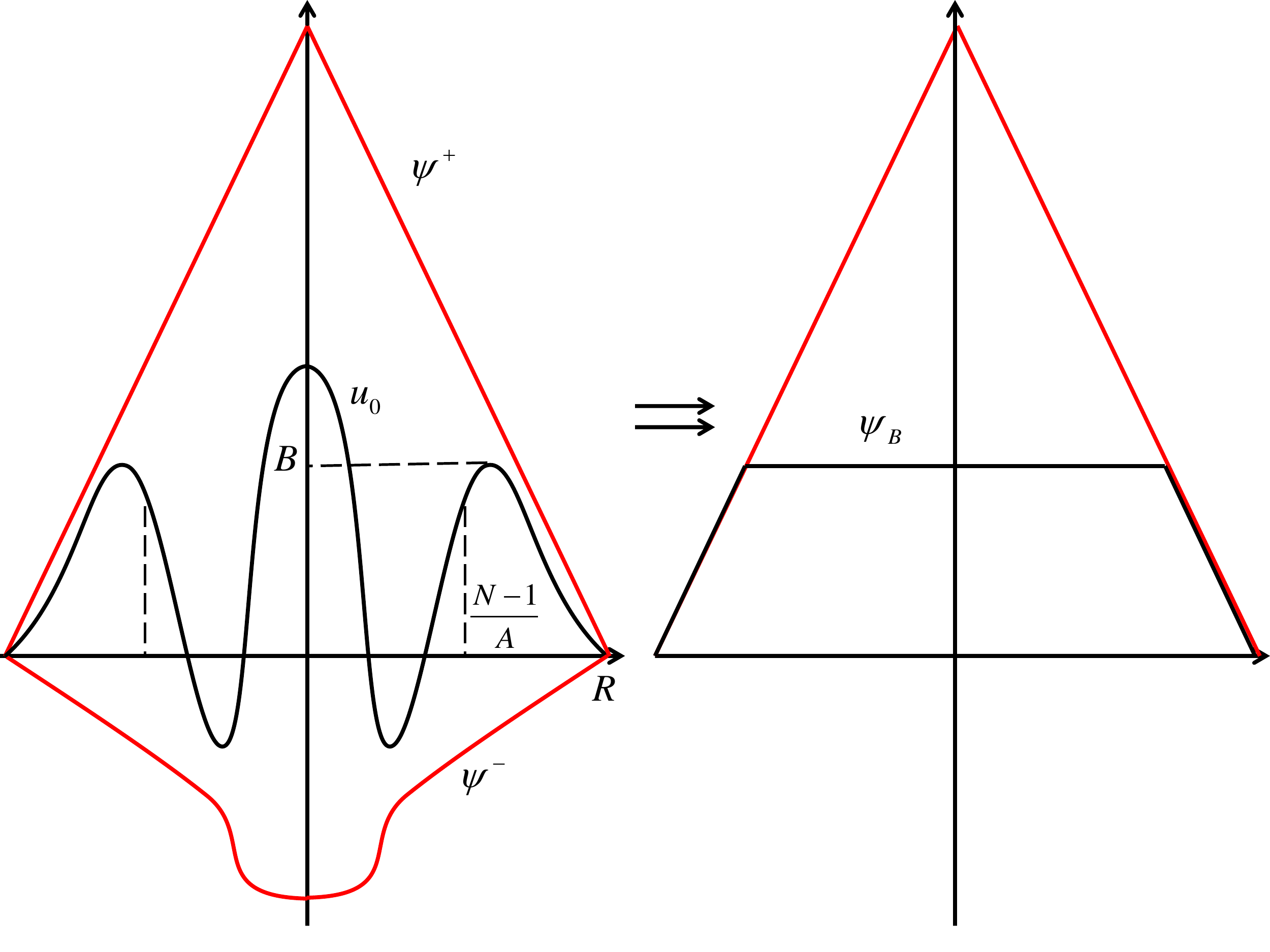}
		\vskip 0pt
		\caption{Case (2) of Theorem \ref{thm:asymcircle}}
        \label{fig:maintheorem}
	\end{center}
\end{figure}

\begin{rem}\label{rem:stationary}
We have the following observations.

\noindent (1) Obviously, $\psi^+(x)$ is radial symmetric.

\noindent (2) In Theorem \ref{thm:asymcircle}, noting that $u_0=0$ on $\partial B_R(O)$, and $u_0\leq \psi^+$, we have $0\leq B\leq \lambda(N-1)/A$. 

\noindent (3) We will see that, for all $0\leq C\leq \lambda(N-1)/A$, $\psi_C$ are all (viscosity) subsolutions of (\ref{eq:levelset}) (without obstacle). Moreover, $\psi_C$ are all (viscosity) radially symmetric stationary solutions of (\ref{eq:levelset}), and (\ref{eq:obs}).

\noindent (4) We consider
 \begin{equation}
0=|D u|\text{div}\left(\frac{D u}{|D u|}\right)+A|D u|,\quad \text{in}\ \Omega,\tag{\ref{eq:levelset}$e$}
\end{equation}
the corresponding elliptic problem of (\ref{eq:levelset}). It is not hard to see that the comparison principle for (\ref{eq:levelset}$e$) does not hold as the equation (\ref{eq:levelset}$e$) is not monotone in $u$. Indeed, obviously $0$ is a solution of (\ref{eq:levelset}$e$). As mentioned in (3), for all $0\leq C\leq \lambda(N-1)/A$, $\psi_C$ are (viscosity) subsolutions of (\ref{eq:levelset}$e$) with $\psi_C=0$ on $\partial \Omega$. 
\end{rem}

\subsection{Motivation and Background}

 In 1994, Sternberg and Ziemer \cite{SZ} consider the following problem
$$
\left\{
\begin{array}{lcl}
u_t-|D u|{\rm div}\dis{\left(\frac{D u}{|D u|}\right)}=0,\quad (x,t)\in\Omega\times(0,T),\\
u(x,0)=u_0(x),\quad x\in \overline{\Omega},\\
u(x,t)=g(x),\quad (x,t)\in\partial \Omega\times[0,T).
\end{array}
\right.
$$
Under the assumption that domain $\Omega$ is mean convex, they show that the solution exists globally in time, is unique, and 
\begin{equation}\label{eq:estimatelip}
\left\Vert u\right\Vert_{C^{0,1}(\overline{\Omega}\times[0,\infty))}\leq C.
\end{equation}
Moreover, they also obtain large time behavior result of $u(x,t)$ as $t\rightarrow\infty$.  We are tempting to derive similar results for generalized motion by mean curvature with driving force. However, global estimate (\ref{eq:estimatelip}) may fail for the solution of 
\begin{equation}
\left\{
\begin{array}{lcl}
u_t-|D u|{\rm div}\dis{\left(\frac{D u}{|D u|}\right)}-A|D u|=0,\quad (x,t)\in\Omega\times(0,T),\\
u(x,0)=u_0(x),\quad x\in \overline{\Omega},\\
u(x,t)=g(x),\quad (x,t)\in\partial \Omega\times[0,T),
\end{array}
\right.\tag{*}
\end{equation}
if the boundary condition is fulfilled in classical sense. For instance, let $\Omega=B_2(O)\subset \mathbb{R}^2$, $g(x)=0$ on $\partial \Omega$, and $A=1$. Consider 
$$
\underline{\psi}(x,t)=\left\{
\begin{array}{lcl}
1,\quad |x|\leq 2-\frac{1}{2}e^{-\frac{1}{6}t},\\
2e^{\frac{1}{6}t}(2-|x|),\quad 2-\frac{1}{2}e^{-\frac{1}{6}t}<|x|\leq 2.\\
\end{array}
\right.
$$ 
As we will see in Appendix, $\underline{\psi}$ is a subsolution of (*), and satisfies
$$
\left\Vert\underline{\psi}\right\Vert_{C^{0,1}(\overline{\Omega}\times[0,\infty))}=\infty.
$$
Therefore, as long as $u_0\geq \underline{\psi}(\cdot,0)$, by the comparison principle, the solution $u$ also satisfies
$$
\left\Vert u\right\Vert_{C^{0,1}(\overline{\Omega}\times[0,\infty))}=\infty,
$$
 provided that the boundary value of $u$ agrees with $g$.

For these reasons, we study the problem with obstacle instead of Dirichlet problem. In 2014, Mercier and  Novaga \cite{MN} study the mean curvature flow with obstacle in classical sense. In 2016, Mercier  \cite{M} gives the well-posed result for the problem (\ref{eq:levelset})--(\ref{eq:obs}) in the viscosity sense. In this research, they prove the comparison principle and give the existence, uniqueness results. We introduce them in Section 2.

\noindent {\bf Background.} It is expected that a proper understanding of the Dirichlet problem is an obstacle formulation. Consider a curve evolving by the forced curvature flow equation $V=-\kappa+A$, where $V$ is the normal velocity and $-\kappa$ is the curvature in the direction of the normal. Suppose both ends are fixed at the points $P_1$ and $P_2$. We take an obstacle functions $\psi^+ \geq 0 \geq \psi^-$ such that it vanishes only on $P_1$ and $P_2$. Let us explain naively. We denote the curve a part of the boundary of $\{u>0\}$ which consists of at least two connected component ``front" and ``back". Then the front level set of solution of (1.1) is expected to give a solution of Dirichlet problem for $V = -\kappa+1$. The difficulty of the Dirichlet problem is that the curve does not divide into two parts. This is a reason we mention ``front" and ``back" of the level-set. Such a problem has been arisen when one discusses spiral growths. In \cite{OTG1, OTG2}, a spiral growth by $V=-\kappa+1$ is discussed for the Neumann boundary condition by using a modified level-set method. It seems to be possible to discuss the Dirichlet problem by using this obstacle approach.

 The level set method for mean curvature flow in viscosity solution frame work was developed independently by Chen, Giga and Goto \cite{CGG}, and Evans, Spruck \cite{ES}. They prove the viscosity solution for level set method exists and is unique. Recently, there are some researches considering the mean curvature flow with driving force. In 2016, Giga, Mitake, and Tran \cite{GMT} consider a crystal growth phenomenon in both vertical and horizontal directions. Indeed, the horizontal direction growth is our mean curvature flow with driving force; see also \cite{ICM}, \cite{GMOT} and \cite{GMT} for a survey and more developments. In our case here, there is no source term, hence, no vertical growth. In 2017, Zhang consider the mean curvature flow with driving force by level set method and give some criteria to judge whether the zero set is fattening or not (see \cite{Z, Z2}).

This paper is organized as following. In Section 2, we give the notion of viscosity solutions to the obstacle problem and some basic results. In Section 3, we prove the gradient estimates and give the large time behavior result. In Section 4, we give a full characterization of the limiting profile in the radially symmetric setting. To the best of our knowledge, this result is new in the literature. It is still an open problem on analyzing the limit in general setting.  

\section{Preliminaries}\large

In this section, we introduce the notion of viscosity solution to the obstacle problem (\ref{eq:levelset})--(\ref{eq:obs}) and give some related results.

 Let $F:\mathbb{R}^N\setminus \{0\}\times \mathcal{S}_N\rightarrow \mathbb{R}$ ($\mathcal{S}_N$ is the set of square symmetric matrices of size  $N$) be such that
$$
F(p,X)=\text{trace}\left(\left(I-\frac{p\otimes p}{|p|^2}\right)X\right)+A|p|.
$$
Denote by, for $(p,X)\in\mathbb{R}^N\times\mathcal{S}_N$,
$$
F_*(p,X)=\liminf_{(q,Y)\rightarrow(p,X)}F(q,Y),
$$
and
$$
F^*(p,X)=\limsup_{(q,Y)\rightarrow(p,X)}F(q,Y).
$$
The above limits are $(q,Y)\rightarrow(p,X)$ in $\mathbb{R}^N\times\mathbb{R}^{N^2}$.

\begin{defn} A function $u:\mathbb{R}^N\times [0,\infty)\rightarrow \mathbb{R}$ is said to be a (viscosity) subsolution of the problem (\ref{eq:levelset})--(\ref{eq:obs}) if

$\cdot$ $u$ is upper semicontinuous (usc),

$\cdot$ for all $(x,t)\in\mathbb{R}^N\times[0,T)$, $\psi^-(x)\leq u(x,t)\leq \psi^+(x)$,

$\cdot$ for all $x\in \mathbb{R}^N$, $u(x,0)\leq u_0(x)$,

$\cdot$ for smooth function $\varphi$, if $(\hat{x},\hat{t})\in \mathbb{R}^N\times(0,T)$ is a maximizer of $u-\varphi$, and $u(\hat{x},\hat{t})>\psi^-(\hat{x})$, then, at $(\hat{x},\hat{t})$,
$$
\varphi_t-F^*(D\varphi,D^2\varphi)\leq 0.
$$

\noindent Similarly, a function $u:\mathbb{R}^N\times [0,\infty)\rightarrow \mathbb{R}$ is said to be a (viscosity) supersolution of the problem (\ref{eq:levelset})--(\ref{eq:obs}) if

$\cdot$ $u$ is lower semicontinuous (usc),

$\cdot$ for all $(x,t)\in\mathbb{R}^N\times[0,T)$, $\psi^-(x)\leq u(x,t)\leq \psi^+(x)$,

$\cdot$ for all $x\in \mathbb{R}^N$, $u(x,0)\geq u_0(x)$,

$\cdot$ for smooth function $\varphi$, if $(\hat{x},\hat{t})\in \mathbb{R}^N\times(0,T)$ is a minimizer of $u-\varphi$, and $u(\hat{x},\hat{t})<\psi^+(\hat{x})$, then, at $(\hat{x},\hat{t})$,
$$
\varphi_t-F_*(D\varphi,D^2\varphi)\geq 0.
$$

\noindent Finally, $u$ is said to be a (viscosity) solution of (\ref{eq:levelset})--(\ref{eq:obs}) if $u$ is both a viscosity subsolution and a supersolution.
\end{defn}

\begin{prop}[Comparison principle]\label{prop:comp}
We assume $u$ is a subsolution and $v$ is a supersolution of {\rm(\ref{eq:levelset})--(\ref{eq:obs})}, respectively. Then $u\leq v$ in $\mathbb{R}^N\times(0,T)$.
\end{prop}

\begin{thm}[Well-posedness]\label{thm:exuni}
Under our assumptions, {\rm(\ref{eq:levelset})--(\ref{eq:obs})} has a unique solution.
\end{thm}

The comparison principle and well-posedness of (\ref{eq:levelset})--(\ref{eq:obs}) are quite standard. We refer to \cite{M}.
To derive convergence results, it is convenient to consider the approximate problem of (\ref{eq:levelset})--(\ref{eq:obs}) by considering, for $\varepsilon>0$, $T>0$,

\begin{equation}
u^{\varepsilon}_t=\sqrt{\varepsilon^2+|D u^{\varepsilon}|^2}\,\text{div}\left(\frac{D u^{\varepsilon}}{\sqrt{\varepsilon^2+|D u^{\varepsilon}|^2}}\right)+A\sqrt{\varepsilon^2+|D u^{\varepsilon}|^2},\quad \text{in}\ \mathbb{R}^N\times(0,T),\tag{\ref{eq:levelset}$\varepsilon$}
\end{equation}
\begin{equation}
u^{\varepsilon}(x,0)=u_0(x),\quad \text{on}\ \mathbb{R}^N,\tag{\ref{eq:initial}$\varepsilon$}
\end{equation}
\begin{equation}
(\psi^{\varepsilon})^-(x)\leq u^{\varepsilon}(x,t)\leq (\psi^{\varepsilon})^+(x),\quad \text{on}\ \mathbb{R}^N\times[0,T).\tag{\ref{eq:obs}$\varepsilon$}
\end{equation}
Here $(\psi^{\varepsilon})^{\pm}$ are smooth, $\rm{supp}(\psi^{\varepsilon})^{\pm}\subset \Omega^{\varepsilon}$, and $\psi^{\varepsilon \pm}\rightarrow \psi^{\pm}$ uniformly in $\mathbb{R}^N$, where
$$
\Omega^{\varepsilon}=\{x\in\mathbb{R}^N\mid d(x,\Omega)<\varepsilon\}.
$$
Moreover, assume $(\psi^{\varepsilon})^{\pm}$ are $L_{\varepsilon}$-Lipschitz continuous, and $L_{\varepsilon}\leq K$ for some constant $K$.

Let $v^{\varepsilon}=u^{\varepsilon}/{\varepsilon}$. Then $v^{\varepsilon}$ satisfies

\begin{equation}\label{eq:graf}
v^{\varepsilon}_t=\sqrt{1+|D v^{\varepsilon}|^2}\,\text{div}\left(\frac{D v^{\varepsilon}}{\sqrt{1+|D v^{\varepsilon}|^2}}\right)+A\sqrt{1+|D v^{\varepsilon}|^2},\quad \text{in}\  \mathbb{R}^N\times(0,T),
\end{equation}
\begin{equation}\label{eq:initialgra}
v^{\varepsilon}(x,0)=u_0(x)/\varepsilon,\quad \text{on}\ \mathbb{R}^N,
\end{equation}
\begin{equation}\label{eq:obsgra}
(\psi^{\varepsilon})^-(x)/\varepsilon \leq v^{\varepsilon}(x,t)\leq (\psi^{\varepsilon})^+(x)/\varepsilon,\quad \text{in}\ \mathbb{R}^N\times[0,T).
\end{equation}

\begin{prop}\label{prop:compgra}
We assume $u$ is a subsolution and $v$ is a supersolution of {\rm(\ref{eq:graf})--(\ref{eq:obsgra})}, respectively, and $u(\cdot,0)\leq v(\cdot,0)$ in $\mathbb{R}^N$. Then $u\leq v$ in $\mathbb{R}^N\times(0,T)$.
\end{prop}

\begin{prop}\label{prop:exuniquegra}
Problem {\rm(\ref{eq:graf})--(\ref{eq:obsgra})} has a unique continuous solution $v^{\varepsilon}$, and furthermore, $v^{\varepsilon}\in C(\mathbb{R}^N\times[0,T))\cap C^{1,1}(\mathbb{R}^N\times(0,T))$.
\end{prop}

In \cite{MN}, Mercier and Novaga give the well-posedness for  problem (\ref{eq:graf})--(\ref{eq:obsgra}) with $A=0$ in the viscosity sense by Perron's method and the usual comparison principle. By repeating these standard arguments, the well-posed result for problem (\ref{eq:graf})--(\ref{eq:obsgra}) with $A>0$ holds. For the regularity, using \cite[Theorem 4.1]{PS}, we deduce that $v^{\varepsilon}\in C^{1,1}(\mathbb{R}^N\times(0,T))$. We omit the details here.

Propositions \ref{prop:compgra} and \ref{prop:exuniquegra} give us the following results immediately.

\begin{thm}\label{thm:compgra}
We assume $u$ is a subsolution and $v$ is a supersolution of {\rm(\ref{eq:levelset}$\varepsilon$)--(\ref{eq:obs}$\varepsilon$)}, respectively. Then $u\leq v$ in $\mathbb{R}^N\times(0,T)$.
\end{thm}

\begin{thm}\label{thm:exuniquegra}
Problem {\rm(\ref{eq:levelset}$\varepsilon$)--(\ref{eq:obs}$\varepsilon$)} has a unique solution $u^{\varepsilon}\in C(\mathbb{R}^N\times[0,T))\cap C^{1,1}(\mathbb{R}^N\times(0,T))$.
\end{thm}

\begin{lem}\label{lem:stable}
Assume $u^{\varepsilon}$ is the unique solution of {\rm (\ref{eq:levelset}$\varepsilon$)--(\ref{eq:obs}$\varepsilon$)} for each $\varepsilon>0$, and there exists $u$ such that
$$
u^{\varepsilon}\rightarrow u,\quad \text{as}\ \varepsilon\rightarrow 0,
$$
uniformly in $\mathbb{R}^N\times[0,T)$.

Then $u$ is the unique viscosity solution of {\rm (\ref{eq:levelset})--(\ref{eq:obs})}.
\end{lem}

\begin{proof}

We first show that $u$ is a supersolution. Let $\varphi$ be a smooth test function such that, at some point $(\hat{x},\hat{t})\in\mathbb{R}^N\times(0,T)$, $u(\hat{x},\hat{t})<\psi^+(\hat{x})$, and 
$$
(u-\varphi)(\hat{x},\hat{t})<(u-\varphi)(x,t),\quad (x,t)\in\mathbb{R}^N\times(0,T)\setminus\{(\hat{x},\hat{t})\}.
$$
Then there exists a neighborhood $V\subset\mathbb{R}^N\times[0,T)$ of $(\hat{x},\hat{t})$ such that, for $(x,t)\in V$,
\begin{equation}\label{eq:stales}
u^{\varepsilon}(x,t)< (\psi^{\varepsilon})^+(x).
\end{equation}
Assume 
$$
\min\limits_{\mathbb{R}^N\times(0,T)}(u^{\varepsilon}-\varphi)=(u^{\varepsilon}-\varphi)(x^{\varepsilon},t^{\varepsilon}).
$$
By a standard argument (\cite[Lemma 2.2.5]{G}), $(x^{\varepsilon},t^{\varepsilon})\rightarrow(\hat{x},\hat{t})$, as $\varepsilon\rightarrow0$, by passing to a subsequence if necessary. Thanks to (\ref{eq:stales}), for $\varepsilon>0$ small enough, the viscosity supersolution test gives that
$$
\left(\varphi_t-\sqrt{\varepsilon^2+|D \varphi|^2}\,\text{div}\left(\frac{D \varphi}{\sqrt{\varepsilon^2+|D\varphi|^2}}\right)+A\sqrt{\varepsilon^2+|D \varphi|^2}\big)\right)(x^{\varepsilon},t^{\varepsilon})\geq 0.
$$
Letting $\varepsilon\rightarrow 0$, we have
$$
\big(\varphi_t-F_*(D\varphi,D^2\varphi)\big)(\hat{x},\hat{t})\geq 0.
$$
The proof of subsolution property is similar to the above, and hence, is omitted.
\end{proof}

\section{Lipschitz bounds and large time profiles}\large

In this section, we prove Theorems \ref{thm:grad} and \ref{thm:asym}. 

\begin{proof}[Proof of Theorem \ref{thm:grad}] The proof of spatial Lipschitz bounds is a simple adjustment of that without obstacle; see e.g., \cite{G}.

Denote 
$$
u_z^{+}(x,t)=(u(x+z,t)+L|z|)\wedge\psi^{+}(x),
$$
$$
u_z^{-}(x,t)=(u(x+z,t)-L|z|)\vee\psi^{-}(x),
$$
for $x,z\in\mathbb{R}^N$, $t\in[0,T)$. We claim $u_z^{+}$ is a supersolution and $u_z^{-}$ is a subsolution of (\ref{eq:levelset})--(\ref{eq:obs}), respectively. Once we have this claim, Proposition \ref{prop:comp} shows that 
$$
u_z^{-}(x,t)\leq u(x,t)\leq u_z^{+}(x,t),\quad\ \text{for\ all}\ (x,t)\in\mathbb{R}^N\times(0,T). 
$$
Thus,
$$
u(x+z,t)-L|z|\leq u(x,t)\leq u(x+z,t)+L|z|.
$$
Consequently, for every $x,y\in\mathbb{R}^N$, $t\in[0,T)$,
$$
|u(x,t)-u(y,t)|\leq L|x-y|.
$$

We only prove the claim for $u_z^{+}$. First, we note 
$$
u(x+z,t)+L|z|\geq \psi^-(x+z)+L|z|\geq \psi^-(x).
$$
Consequently, 
\begin{equation}\label{eq:obstasup}
\psi^-(x)\leq u_z^{+}(x,t)\leq \psi^+(x),\quad\ \text{for\ all}\ (x,t)\in\mathbb{R}^N\times[0,T).
\end{equation}
Since the initial data $u_0$ is $L$-Lipschitz,
\begin{equation}\label{eq:initialsup}
u(x+z,0)+L|z|=u_0(x+z)+L|z|\geq u_0(x).
\end{equation}
Then we get $u_z^+(\cdot,0)\geq u_0$ in $\mathbb{R}^N$. Obviously, $w_z^{+}(x,t):=u(x+z,t)+L|z|$ satisfies equation (\ref{eq:levelset}). Combining (\ref{eq:obstasup}), (\ref{eq:initialsup}), and the definition of viscosity supersolution, $u_z^{+}$ is a viscosity supersolution of (\ref{eq:levelset})--(\ref{eq:obs}).

Next we consider 
$$
u_s^+(x,t)=(u(x,t+s)+Ms)\wedge\psi^+,
$$
$$
u_s^-(x,t)=(u(x,t+s)-Ms)\vee\psi^-,
$$
$(x,t)\in \mathbb{R}^N\times[0,s]$, for $0\leq s<T$. We claim $u_s^+$ is a supersolution and $u_s^-$ is a subsolution of (\ref{eq:levelset})--(\ref{eq:obs}), respectively. If we have this claim, then by the comparison principle,
$$
u_s^-\leq u\leq u_s^+.
$$
Consequently, 
$$
u(x,t+s)-Ms\leq u(x,t)\leq u(x,t+s)+Ms.
$$
Therefore, 
$$
|u(x,t)-u(x,r)|\leq M|t-r|,
$$
for all $x\in\mathbb{R}^N$, $t,r\in(0,T)$. 

Next we only prove the claim for $u_s^+$. Seeing the choice of $M$, it is easy to prove $u_1(x,t)=(u_0(x)-M t)\vee\psi^-$ is a subsolution of (\ref{eq:levelset})--(\ref{eq:obs}). By the comparison principle,
$$
u(x,t)\geq u_1(x,t)\geq u_0(x)-Mt.
$$
This implies 
\begin{equation}\label{eq:timevinitial}
u_s^{+}(x,0)=(u(x,s)+Ms)\wedge\psi^+\geq u_0(x).
\end{equation}
Obviously, 
\begin{equation}\label{eq:timevobs}
\psi^+\geq  u_s^{+}(x,t)=(u(x,t+s)+Ms)\wedge\psi^+\geq \psi^-.
\end{equation}

Note $h_s^{+}(x,t):=u(x,t+s)+Ms$ satisfies equation (\ref{eq:levelset}). Combining (\ref{eq:timevinitial}), (\ref{eq:timevobs}), and the definition of viscosity supersolution, $u_s^{+}$ is a viscosity supersolution of (\ref{eq:levelset})--(\ref{eq:obs}).
\end{proof}
Next we prove Theorem \ref{thm:asym}.
\begin{proof}[Proof of Theorem \ref{thm:asym}]
We divide the proof into four steps.

{\bf Step 1.} $u^{\varepsilon}$ are Lipschitz continuous for all $\varepsilon \in (0,1)$. Moreover,
$$
|u^{\varepsilon}(x,t)-u^{\varepsilon}(y,t)|\leq L_{\varepsilon}|x-y|\leq K|x-y|,
$$
and
$$
|u^{\varepsilon}(x,t)-u^{\varepsilon}(x,r)|\leq M|t-r|,
$$
for $x,y\in\mathbb{R}^N$, $t,r\in(0,T)$. 

By constructing subsolution and supersolution as in Theorem \ref{thm:grad}, we can prove these results easily. We leave the details to the readers.

{\bf Step 2.} There exists constant $C>0$ independent of $\varepsilon$ and $T$ such that
\begin{equation}\label{eq:intebound}
\int_0^T\int_{\mathbb{R}^N}(u_t^{\varepsilon})^2\,dxdt\leq C.
\end{equation}

We consider the following Lyapunov function (see e.g., \cite{GOS})
$$
I^{\varepsilon}(t)= \int_{\mathbb{R}^N}\sqrt{\varepsilon^2+|D u^{\varepsilon}|^2}\,dx.
$$
By calculation,
\begin{eqnarray*}
\frac{{\rm d}}{{\rm d}t}\int_{\mathbb{R}^N}\sqrt{\varepsilon^2+|D u^{\varepsilon}|^2}\,dx=\int_{\mathbb{R}^N}\frac{D u^{\varepsilon}\cdot D u_t^{\varepsilon}}{\sqrt{\varepsilon^2+|D u^{\varepsilon}|^2}}\,dx=-\int_{\mathbb{R}^N}u_t^{\varepsilon}{\rm div}\left(\frac{D u^{\varepsilon}}{\sqrt{\varepsilon^2+|D u^{\varepsilon}|^2}}\right)\,dx.
\end{eqnarray*}
If $u^{\varepsilon}(x_0,t_0)=(\psi^{\varepsilon})^+(x_0)$ at some $(x_0,t_0)\in\mathbb{R}^N\times(0,T)$, it also means
$$
\max\limits_{(x,t)\in\mathbb{R}^N\times(0,T)}(u^{\varepsilon}(x,t)-(\psi^{\varepsilon})^+(x))=u^{\varepsilon}(x_0,t_0)-(\psi^{\varepsilon})^+(x_0)=0.
$$
then there holds $u_t^{\varepsilon}(x_0,t_0)=0$. Same claim holds if $u^{\varepsilon}(x_0,t_0)=(\psi^{\varepsilon})^-(x_0)$. Consequently,
$$
F^{\pm}(t)=\{x\in\mathbb{R}^N\mid u^{\varepsilon}(x,t)=(\psi^{\varepsilon})^{\pm}(x),\, u_t^{\varepsilon}(x,t)\neq 0\}=\emptyset,
$$
for $t\in(0,T)$. Let $$
Q^{\pm}(t)=\{x\in\mathbb{R}^N\mid u^{\varepsilon}(x,t)\neq(\psi^{\varepsilon})^{\pm}(x) \},\quad t\in(0,T).
$$
Note the fact that for $x\in Q^+(t)\cap Q^-(t)$, we have
$$
u^{\varepsilon}_t=\sqrt{\varepsilon^2+|D u^{\varepsilon}|^2}\,\text{div}\left(\frac{D u^{\varepsilon}}{\sqrt{\varepsilon^2+|D u^{\varepsilon}|^2}}\right)+A\sqrt{\varepsilon^2+|D u^{\varepsilon}|^2}
$$
at $(x,t)$. Therefore,
\begin{eqnarray*}
\frac{{\rm d}}{{\rm d}t}\int_{\mathbb{R}^N}\sqrt{\varepsilon^2+|D u^{\varepsilon}|^2}\,dx&=&-\int_{\mathbb{R}^N\setminus (F^+\cup F^-)}u_t^{\varepsilon}{\rm div}\left(\frac{D u^{\varepsilon}}{\sqrt{\varepsilon^2+|D u^{\varepsilon}|^2}}\right)\,dx\\
&=&-\int_{Q^+\cap Q^-}u_t^{\varepsilon}{\rm div}\left(\frac{D u^{\varepsilon}}{\sqrt{\varepsilon^2+|D u^{\varepsilon}|^2}}\right)\,dx\\
&=&-\int_{Q^+\cap Q^-}\left(\frac{(u_t^{\varepsilon})^2}{\sqrt{\varepsilon^2+|D u^{\varepsilon}|^2}}-Au_t^{\varepsilon}\right)\,dx\\
&=&-\int_{\mathbb{R}^N}\left(\frac{(u_t^{\varepsilon})^2}{\sqrt{\varepsilon^2+|D u^{\varepsilon}|^2}}-Au_t^{\varepsilon}\right)\,dx\\
&\leq&-\frac{1}{\sqrt{\varepsilon^2+K^2}}\int_{\mathbb{R}^N}(u_t^{\varepsilon})^2dx+A\int_{\mathbb{R}^N}u_t^{\varepsilon}\,dx.
\end{eqnarray*}
 Then, 
$$
\frac{\rm d}{{\rm d}t}\left(\int_{\mathbb{R}^N}\sqrt{\varepsilon^2+|D u^{\varepsilon}|^2}\,dx-A\int_{\mathbb{R}^N}u^{\varepsilon}\,dx\right)\leq-\frac{1}{\sqrt{\varepsilon^2+K^2}}\int_{\mathbb{R}^N}(u_t^{\varepsilon})^2\,dx.
$$
Integrating the inequality above, we have
\begin{eqnarray*}
\int_0^{T}\int_{\mathbb{R}^N}(u_t^{\varepsilon})^2\,dxdt&\leq& A\sqrt{\varepsilon^2+K^2}\int_{\mathbb{R}^N}(u^{\varepsilon}(x,T)-u^{\varepsilon}(x,0))\,dx\\
&+&\sqrt{\varepsilon^2+K^2}\int_{\mathbb{R}^N}\left(\sqrt{\varepsilon^2+|D u^{\varepsilon}|^2(x,0)}-\sqrt{\varepsilon^2+|D u^{\varepsilon}|^2(x,T)}\,\right)\,dx.
\end{eqnarray*}
The assumptions for $(\psi^{\varepsilon})^{\pm}$ show that we can find a constant $C_1$ independent of $\varepsilon \in (0,1)$ such that 
$$
\int_{\mathbb{R}^N}|(\psi^{\varepsilon})^{\pm}|(x)\,dx\leq C_1.
$$
Then, for $\varepsilon\in(0,1)$, 
$$
\int_0^{T}\int_{\mathbb{R}^N}(u_t^{\varepsilon})^2\, dxdt\leq 2AC_1\sqrt{1+K^2}+\sqrt{1+K^2}\int_{\mathbb{R}^N}\sqrt{1+|D u_0^{\varepsilon}|^2(x)}\,dx:=C.
$$
Thus, 
$$
\int_0^{T}\int_{\mathbb{R}^N}(u_t^{\varepsilon})^2\,dxdt\leq C.
$$

{\bf Step 3.} $u^{\varepsilon}_t\rightharpoonup u_t$ weakly in $L^2(\mathbb{R}^N\times[0,\infty))$, as $\varepsilon\rightarrow 0$.

\noindent Step 2 shows that $u^{\varepsilon}_t$ is bounded in $L^2(\mathbb{R}^N\times[0,\infty))$. Since $L^2(\mathbb{R}^N\times[0,\infty))$ is a Hilbert space, there exists $g\in L^2(\mathbb{R}^N\times[0,\infty))$ such that, by passing to a subsequence if needed, 
$$
u^{\varepsilon}_t\rightharpoonup g,
$$
weakly in $L^2(\mathbb{R}^N\times[0,\infty))$, as $\varepsilon\rightarrow 0$. For every $\phi\in C_0^{\infty}(\mathbb{R}^N\times(0,T))$, 
$$
\int_0^{T}\int_{\mathbb{R}^N}u_t^{\varepsilon}\phi \,dxdt\rightarrow \int_0^{T}\int_{\mathbb{R}^N}g\phi \,dxdt,
$$
as $\varepsilon\rightarrow0$, for every $T>0$. On the other hand, Step 1 and Lemma \ref{lem:stable} show that $u^{\varepsilon}\rightarrow u$ uniformly on $\mathbb{R}^N\times[0,T]$  for every $T>0$. Therefore,
$$
\int_0^{T}\int_{\mathbb{R}^N}u_t^{\varepsilon}\phi\, dxdt=-\int_0^{T}\int_{\mathbb{R}^N}u^{\varepsilon}\phi_t \,dxdt\rightarrow-\int_0^{T}\int_{\mathbb{R}^N}u\phi_t \,dxdt.
$$
Then
$$
\int_0^{T}\int_{\mathbb{R}^N}g\phi \,dxdt=-\int_0^{T}\int_{\mathbb{R}^N}u\phi_t \,dxdt
$$
which shows that $g=u_t$ in $\mathbb{R}^N\times[0,\infty)$. Thus, $u^{\varepsilon}_t\rightharpoonup u_t$ weakly in $L^2(\mathbb{R}^N\times[0,\infty))$, as $\varepsilon\rightarrow0$ (whole sequence).

{\bf Step 4.} We complete the proof in this step.

\noindent By weakly lower semi-continuity, 
\begin{equation}\label{eq:lowersemi}
\int_0^{\infty}\int_{\mathbb{R}^N}(u_t)^2\,dxdt\leq \liminf\limits_{\varepsilon\rightarrow 0}\int_0^{\infty}\int_{\mathbb{R}^N}(u_t^{\varepsilon})^2\,dxdt\leq C.
\end{equation}
For every $\{t_k\}\rightarrow\infty$, by the Arzel\`a-Ascoli theorem, there exist a subsequence $\{t_{k_j}\}$ and a Lipschitz continuous function $v$ such that
$$
u_{k_j}(x,t)=u(x,t+t_{k_j})\rightarrow v(x,t),
$$
locally uniformly on $\mathbb{R}^N\times[0,\infty)$. As $u(\cdot,t)$ is compactly supported on $\overline{\Omega}$,   
\begin{equation}\label{eq:converg1}
u_{k_j}(x,t)=u(x,t+t_{k_j})\rightarrow v(x,t),
\end{equation}
uniformly on $\mathbb{R}^N\times[0,T]$, for every $T>0$. By stability results of viscosity solutions, $v$ satisfies
\begin{equation}
v_t=|D v|\text{div}\left(\frac{D v}{|D v|}\right)+A|D v|,\quad (x,t)\in \mathbb{R}^N\times(0,\infty),\tag{\ref{eq:levelset}$a$}
\end{equation}
\begin{equation}
\psi^-(x)\leq v(x,t)\leq \psi^+(x),\quad (x,t)\in\mathbb{R}^N\times[0,\infty).\tag{\ref{eq:obs}$b$}
\end{equation}
Thanks to the fact that 
$$
\int_{0}^{\infty}\int_{\mathbb{R}^N}(u_t)^2\,dxdt\leq C<\infty, 
$$
we have
$$
\int_0^1\int_{\mathbb{R}^N}(u_{k_j})_t^2\,dxdt=\int_{t_{k_j}}^{1+t_{k_j}}\int_{\mathbb{R}^N}(u_t)^2\,dxdt\rightarrow0,
$$
as $j\rightarrow\infty$. This shows that 
$$
(u_{k_j})_t\rightharpoonup0,
$$
weakly in $L^2(\mathbb{R}^N\times[0,1])$. On the other hand, (\ref{eq:converg1}) implies that, by passing to a further subsequence if necessary 
$$
(u_{k_j})_t\rightharpoonup v_t,
$$
weakly in $L^2(\mathbb{R}^N\times[0,1])$. Consequently, $v_t=0$ on $\mathbb{R}^N\times[0,1]$. Thus, $v$ is a solution of 
 \begin{equation}
0=|D v|\text{div}\left(\frac{D v}{|D v|}\right)+A|D v|,\quad x\in \mathbb{R}^N,\tag{\ref{eq:levelset}$e$}
\end{equation}
\begin{equation}
\psi^-(x)\leq v(x)\leq \psi^+(x),\quad x\in\mathbb{R}^N.\tag{\ref{eq:obs}}
\end{equation}
As mentioned in Remark \ref{rem:stationary}, the solution of (\ref{eq:levelset}$e$), (\ref{eq:obs}) is not necessary unique. Therefore, $v$ may depend on the choice of subsequence of $\{t_k\}_{k}$.

At last, we prove that $v$ is independent of the choice of subsequence of $\{t_k\}_{k}$. Since $u_{k_j}$ converges uniformly to $v$ on $\mathbb{R}^N\times[0,1]$, for every $\varepsilon>0$ there exists $j$ large enough such that
$$
|u_{k_j}(x,t)-v(x)|<\varepsilon,\quad\ \text{for all} \ (x,t)\in\mathbb{R}^N\times[0,1].
$$
Setting $t=0$, $v(x)-\varepsilon<u_{k_j}(x,0)<v(x)+\varepsilon$ in $\mathbb{R}^N$. By the comparison principle, 
$$
v(x)-\varepsilon\leq (v(x)-\varepsilon)\vee\psi^{-} \leq u(x,t)\leq (v(x)+\varepsilon)\wedge\psi^{+}\leq v(x)+\varepsilon,
$$
for $(x,t)\in\mathbb{R}^N\times[t_{k_j},\infty)$. This implies that $u(\cdot,t)$ converges uniformly to $v$ in $\mathbb{R}^N$ without taking a subsequence.

\end{proof}
\begin{rem}
The proof of large time behavior is quite standard. Nevertheless, as we do not have uniqueness of solutions to (\ref{eq:levelset}$e$), (\ref{eq:obs}), it is not easy to analyze what is the limiting profile given the initial data $u_0$, and obstacles $\psi^\pm$. In the following, we are able to characterize this in the radially symmetric setting.
\end{rem}

\section{The radially symmetric setting}\large
In this section, we prove Theorem \ref{thm:asymcircle}. First, we find radially symmetric solutions of (\ref{eq:levelset}$e$), and (\ref{eq:obs}). To establish this, we need the following lemma (\cite[Lemma 8.1]{GMT}).

\begin{lem}\label{lem:sub}
Let $\psi:[0,\infty)\rightarrow \mathbb{R}$ be a continuous function, which is $C^2$ in $(0,R)\cup(R,R_1)$, and $\psi=0$ in $[R_1,\infty)$, for some given $R,\ R_1>0$ with $R_1>R$. Assume further that
$$
\psi^{\prime}(R-)=a\quad \text{and}\quad \psi^{\prime}(R+)=b.
$$
The following holds.

(i) If $a<b$, then for any $\phi\in C^2(\mathbb{R}^N)$ such that $\psi(|x|)-\phi(x)$ has a strict minimum at $x_0\in \partial B(0,R)$, then for some $s\in[a,b]$,
$$
D\phi(x_0)=s\frac{x_0}{R}\quad \text{and}\quad {\rm tr}\left[\left((I-\frac{Du\otimes Du}{|Du|^2})\right)D^2 u \right]\leq \frac{(N-1)s}{R}.
$$

(ii) If $a>b$, then for any $\phi\in C^2(\mathbb{R}^N)$ such that $\psi(|x|)-\phi(x)$ has a strict maximum at $x_0\in \partial B(0,R)$, then for some $s\in[b,a]$,
$$
D\phi(x_0)=s\frac{x_0}{R}\quad \text{and}\quad {\rm tr}\left[\left((I-\frac{Du\otimes Du}{|Du|^2})\right)D^2 u \right]\geq \frac{(N-1)s}{R}.
$$
\end{lem}

We always assume that the hypotheses of Theorem \ref{thm:asymcircle} are in force in this section. 
Let us recall them here for clarity
\\(1) $\Omega=B_R(O)$ for given $R>0$;
\\
(2)  
$$
\psi^{+}(x)=\left\{
\begin{array}{lcl}
 \lambda {\rm dist}(x,\partial \Omega)=\lambda(R-|x|),\quad x\in\Omega,\\
0,\quad x\in \mathbb{R}^N\setminus \Omega,
\end{array}
\right.
$$
where $\lambda >0$ is given;
\\
(3) $\psi^-$ is radial, and $\psi^-<0$ in $\Omega$. The initial data $u_0$ is radial, and $\psi^- \leq u_0 \leq \psi^+$.

\begin{prop}\label{prop:stationpositive}
Assume $R\geq (N-1)/A$. The solution $v$ of {\rm (\ref{eq:levelset}$e$), (\ref{eq:obs})} satisfies
$$
0\leq v\leq \psi^{+}\wedge\lambda\left(R-\frac{N-1}{A}\right).
$$
\end{prop}
\begin{proof}
Let 
$$
u_*=\varphi_*\vee\psi^-,
$$
where
$$
\varphi_*(x,t)=\left\{
\begin{array}{lcl}
-Le^{-\gamma t}(R-|x|),\quad |x|\leq R,\, t\geq0,\\
0,\quad |x|>R,\, t\geq 0,
\end{array}
\right.
$$
where $\gamma$ is chosen such that
$$
0<\gamma<\frac{1}{R}\left(\frac{N-1}{R}+A\right).
$$
At $(x_0,t_0)$, for $0<|x_0|<R$, and $t_0>0$,
\begin{multline*}
\left((\varphi_{*})_t-|D \varphi_*|{\rm div}\left(\frac{D \varphi_*}{|D \varphi_*|}\right)-A|D \varphi_*|\right)(x_0,t_0)=L e^{-\gamma t_0} \left(\gamma(R-|x_0|)-\frac{N-1}{|x_0|}-A\right)\\
\dis{\leq L e^{-\gamma t_0} \left(\gamma R-\frac{N-1}{R}-A\right)<0.}
\end{multline*}
Obviously, $u_*(\cdot,0)= \psi^-\leq v$ and $\psi^-\leq u_*\leq \psi^+$. Thus, $u_*$ is subsolution of  (\ref{eq:levelset}), (\ref{eq:initial}), and (\ref{eq:obs}). By the comparison principle, 
$$
u_*(x,t)\leq v(x),\quad\ \text{for\ all}\ (x,t)\in\mathbb{R}^N\times[0,\infty).
$$
Since $u_*(\cdot,t)\rightarrow0$, as $t\rightarrow\infty$, we conclude that $v\geq0$.

For small $\varepsilon>0$, let
$$
u^{*\varepsilon}(x,t)=\left\{
\begin{array}{lcl}
\lambda e^{-\mu t}\left(\frac{N-1}{A}-\varepsilon-|x|\right)+\lambda\left(R-\frac{N-1}{A}+\varepsilon\right),\quad |x|\leq \frac{N-1}{A}-\varepsilon,\, t\geq0,\\
\psi^+,\quad |x|>\frac{N-1}{A}-\varepsilon,\, t\geq 0.
\end{array}
\right.
$$
Here, $0<\mu<\frac{A}{N-1}\left(\frac{N-1}{(N-1)/A-\varepsilon}-A\right).$
\begin{figure}[htbp]
	\begin{center}
            \includegraphics[height=8.0cm]{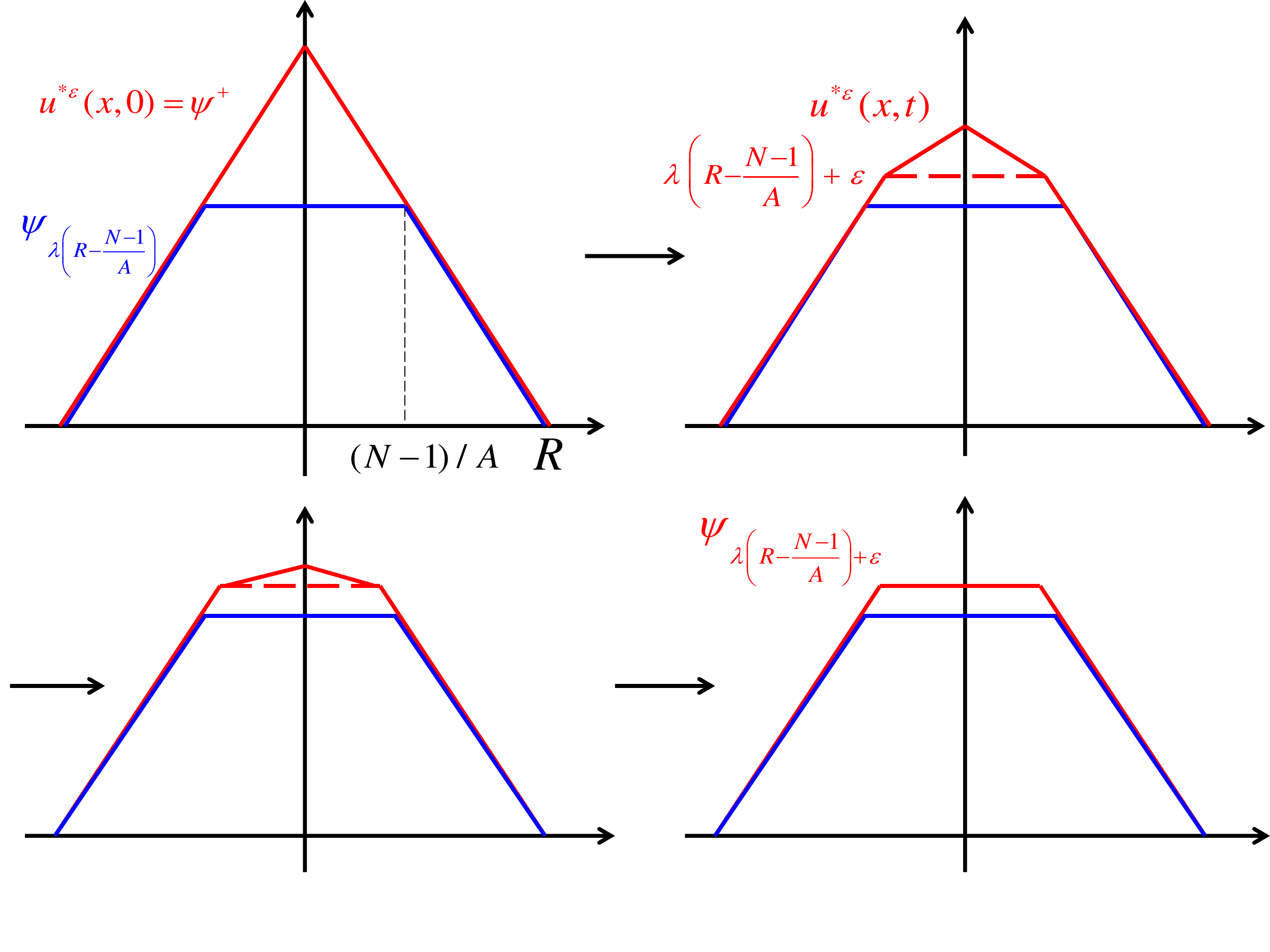}
		\vskip 0pt
		\caption{Supersolution $u^{*\varepsilon}$}
        \label{fig:super1}
	\end{center}
\end{figure}

\noindent At $(x_0,t_0)$, where $0<|x_0|<\frac{N-1}{A}-\varepsilon$, and $t_0>0$, we compute
\begin{align*}
&\left(u^{*\varepsilon}_{t}-|D u^{*\varepsilon}|{\rm div}\left(\frac{D u^{*\varepsilon}}{|D u^{*\varepsilon}|}\right)-A|Du^{*\varepsilon}|\right)(x_0,t_0)\\
=\ &\lambda e^{-\mu t_0} \left(-\mu(\frac{N-1}{A}-\varepsilon-|x_0|)+\frac{N-1}{|x_0|}-A\right)\\
\geq\ & \lambda e^{-\mu t_0}\left(-\mu\frac{N-1}{A}+\frac{N-1}{(N-1)/A-\varepsilon}-A\right)>0.
\end{align*}
Thus, $u^{*\varepsilon}$ is supersolution of  (\ref{eq:levelset}), (\ref{eq:initial}), and (\ref{eq:obs}). 

Obviously, $u^{*\varepsilon}(\cdot,0)=\psi^+\geq v$, and $\psi^-\leq u^{*\varepsilon}\leq \psi^+$. By the comparison principle,
$$
v(x)\leq u^{*\varepsilon},\quad\text{for\ all}\ (x,t)\in\mathbb{R}^N\times[0,\infty).
$$
Since 
$$
u^{*\varepsilon}(\cdot,t)\rightarrow\psi^+\wedge\lambda(R-\frac{N-1}{A}+\varepsilon),
$$
we deduce, as $t\rightarrow\infty$, 
$$
v\leq \psi^+\wedge\lambda(R-\frac{N-1}{A}+\varepsilon).
$$
Letting $\varepsilon\rightarrow0$ in the above to imply
$$
v\leq \psi^+\wedge\lambda(R-\frac{N-1}{A}).
$$
The proof is now complete.
\end{proof}

\begin{prop}\label{prop:symmetricstation}
Assume $R\geq(N-1)/A$. Then
$$
\psi_C:=\psi^+\wedge C,\quad \text{for}\ 0\leq C\leq \lambda\left(R-\frac{N-1}{A}\right),
$$ 
are all radially symmetric and Lipschitz continuous solutions to {\rm (\ref{eq:levelset}$e$)}, and {\rm (\ref{eq:obs})} satisfying that $$
0\leq v\leq \psi^{+}\wedge\lambda\left(R-\frac{N-1}{A}\right).
$$
\end{prop}

\begin{figure}[htbp]
	\begin{center}
            \includegraphics[height=8.0cm]{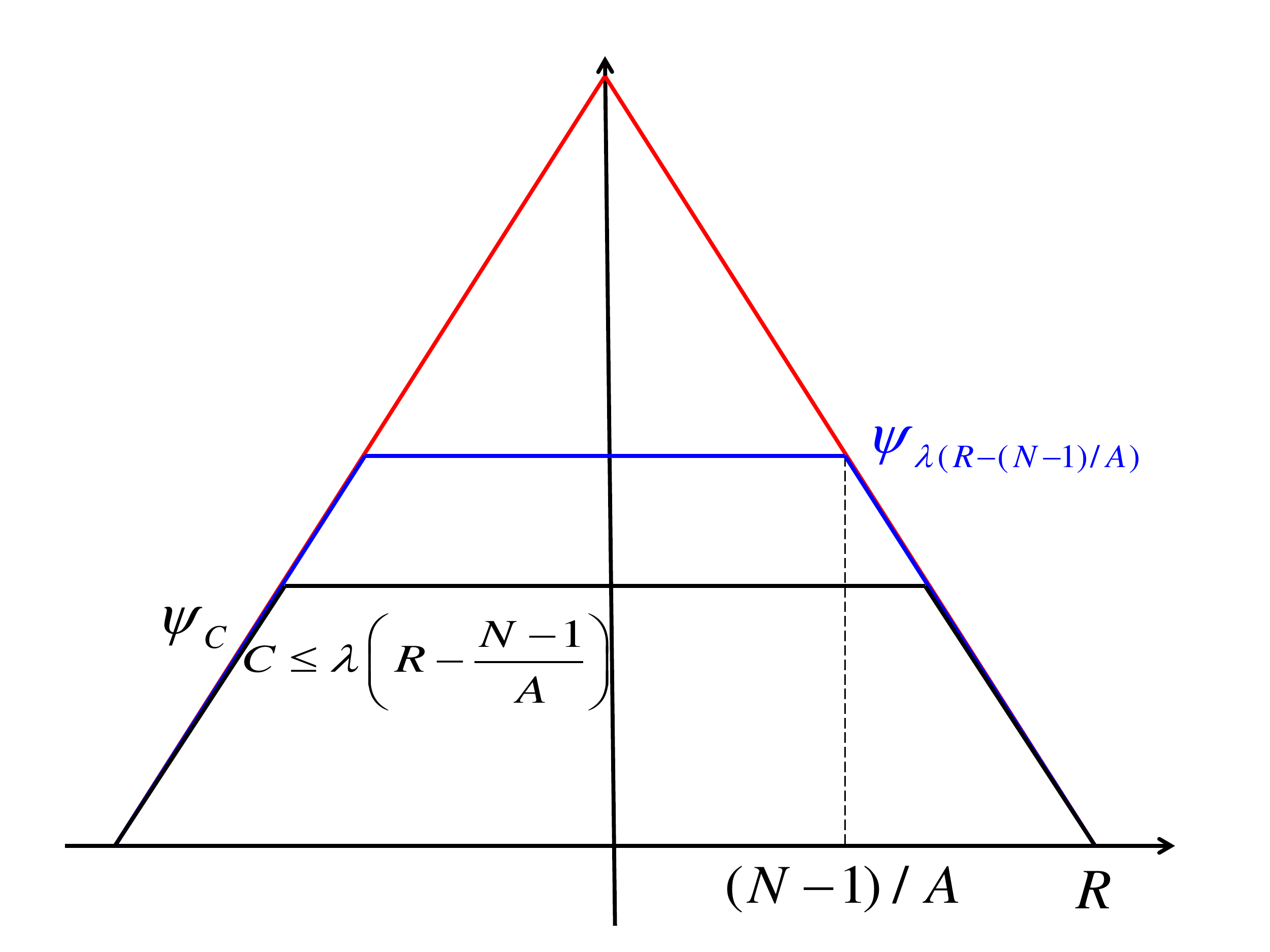}
		\vskip 0pt
		\caption{Stationary solution of (\ref{eq:levelset}$e$), (\ref{eq:obs})}
        \label{fig:stationary}
	\end{center}
\end{figure}
\begin{proof}
Let $v$ be a radially symmetric solution to (\ref{eq:levelset}$e$), and (\ref{eq:obs}). By abuse of notion, we write $v(x)=v(|x|)$ for $x\in\mathbb{R}^N$. Then, $v(r)$ satisfies following ODE
\begin{equation}
\dis{-\frac{(N-1)v^{\prime}}{r}}-A|v^{\prime}|=0,\quad r>0,\tag{\ref{eq:levelset}$r$}
\end{equation}
\begin{equation}
\psi^-(r)\leq v(r)\leq \psi^+(r),\quad r\geq0, \tag{\ref{eq:obs}$r$}
\end{equation}
in the viscosity sense.

Denote by
$$
E=\{r\in(0,R)\mid v(r)=\psi^+(r)\ \text{or}\ v(r)=\psi^-(r)\}.
$$
Obviously, $E$ is a closed set. Then we can find at most a countable number of intervals $(a_i,b_i)$, $i\in \mathbb{N}$, satisfying
$$
b_i\leq a_{i+1},
$$
and
$$
(0,R)\setminus E=\bigcup_{i=1}^{\infty}(a_i,b_i).
$$
Under our assumption, $v(0)<\psi^+(0)$. Thus $a_1=0$.

In each interval $(a_i,b_i)$, $v$ satisfies (\ref{eq:levelset}$r$) in the viscosity sense. If $v$ is differentiable at $r_0\in(a_i,b_i)$, and $r_0\neq (N-1)/A$, there holds
$$
-\frac{(N-1)v^{\prime}(r_0)}{r_0}-A|v^{\prime}(r_0)|=0,
$$
which implies $v^{\prime}(r_0)=0$. Since $v$ is Lipschitz continuous, $v$ is differentiable in $(a_i,b_i)$ almost everywhere. Thus 
$$
v^{\prime}=0,\quad \text{a.e.}\quad \text{in}\ (a_i,b_i),
$$
which gives that $v\equiv c_i$ in $(a_i,b_i)$ as $v$ is Lipschitz continuous. We have 
$$
(a_i,b_i)=\emptyset,
$$
provided that $\psi^+(a_i)>\psi^+(b_i)$, $i\geq 2$. Therefore, $v= c_1$ in $[a_1,b_1)=[0,b_1)$, and $v=\psi^+$ in $[b_1,\infty)$. In other words, 
$$
v=\psi^+\wedge c_1.
$$
Under our assumption, $0\leq c_1\leq\lambda(R-(N-1)/A)$. 

At last we prove that for all $0\leq C\leq\lambda(R-(N-1)/A)$, $\psi^+\wedge C$ is the solution of problem (\ref{eq:levelset}$e$), (\ref{eq:obs}). It is easy to see $\psi^+\wedge C$ is a supersolution. We only show $\psi_C$ is a subsolution. 

This claim is clear for $|x|<R-\frac{C}{\lambda}$ or $|x|\geq R$. We only check carefully where $R-\frac{C}{\lambda}\leq |x_0|<R$.

 For $R-\frac{C}{\lambda}<|x_0|<R$, we have $\psi_C(x_0)=\psi^+(x_0)>\psi^-(x_0)$, and 
$$
\left(-|D \psi^+|{\rm div}\left(\frac{D \psi^+}{|D \psi^+|}\right)-A|D \psi^+|\right)(x_0)=\lambda\left(\frac{N-1}{|x_0|}-A\right)<0.
$$
Here we use the fact $|x_0|>R-\frac{C}{\lambda}\geq (N-1)/A$.

 For $|x_0|=R-\frac{C}{\lambda}$, assume that there is a test function $\varphi\in C^2(\mathbb{R}^N)$ such that $\varphi-\psi_C$ has a strict maximum at $x_0$. By Lemma \ref{lem:sub}, there exists $s\in[-\lambda,0]$ such that
$$
D \varphi(x_0)=\frac{sx_0}{|x_0|}\quad \text{and}\quad {\rm tr}\left[\left((I-\frac{Du\otimes Du}{|Du|^2})\right)D^2 u \right]\geq \frac{(N-1)s}{|x_0|}.
$$
Since $|x_0|=R-\frac{C}{\lambda}\geq (N-1)/A$, and $s\leq 0$, we have
$$
\left(-|D \varphi|{\rm div}\left(\frac{D \varphi}{|D \varphi|}\right)-A|D \varphi|\right)(x_0)\leq -\frac{(N-1)s}{|x_0|}-A|s|=-\frac{(N-1)s}{|x_0|}+As\leq 0.
$$
Thus, $\psi_C$ is subsolution. The proof is complete.
  
\end{proof}

\begin{proof}[Proof of Theorem \ref{thm:asymcircle}]
 If $R=(N-1)/A$, Proposition \ref{prop:symmetricstation} shows that $0$ is the unique solution of (\ref{eq:levelset}$e$), (\ref{eq:obs}). The result is obvious by Theorem \ref{thm:asym}.

 If $R<(N-1)/A$, let 
$$
u^*(x,t)=\left\{
\begin{array}{lcl}
\lambda e^{-\nu_1t}(R-|x|),\quad |x|<R,\\
0,\quad |x|\geq R.
\end{array}
\right.
$$
Here $\nu_1$ is chosen such that
$$
0<\nu_1<\frac{1}{R}\left(\frac{N-1}{R}-A\right).
$$
At $(x_0,t_0)$, where $0<|x_0|<R$, and $t_0>0$,
\begin{multline*}
\left(u^{*}_{t}-|\nabla u^{*}|{\rm div}\left(\frac{\nabla u^{*}}{|\nabla u^{*}|}\right)-A|\nabla u^{*}|\right)(x_0,t_0)=\lambda e^{-\nu_1 t_0} \left(-\nu_1(R-|x_0|)+\frac{N-1}{|x_0|}-A\right)\\
\dis{\geq\lambda e^{-\nu_1 t_0} \left(-\nu_1 R+\frac{N-1}{R}-A\right)>0.}
\end{multline*}
Then $u^*$ is a supersolution. Besides, $u_*$ constructed in Proposition \ref{prop:stationpositive} is a subsolution. By the comparison principle, 
$$
u_*\leq u\leq u^*.
$$
Therefore, $u(\cdot,t)\rightarrow 0$ uniformly in $\mathbb{R}^N$, as $t\rightarrow0$.

We consider the final case where $R>(N-1)/A$.
First we construct a supersolution
$$
\overline{u}^{\varepsilon}(x,t)=\left\{
\begin{array}{lcl}
(B+\varepsilon)\wedge \psi^+,\quad |x|>\frac{N-1}{A}-\frac{\varepsilon}{L},\\
Le^{-\nu_2t}(\frac{N-1}{A}-\frac{\varepsilon}{L}-|x|)+B+\varepsilon,\quad |x|\leq \frac{N-1}{A}-\frac{\varepsilon}{L}.
\end{array}
\right.
$$
Here $\nu_2$ is chosen so that
$$
0<\nu_2<\frac{A}{N-1}\left(\frac{N-1}{\frac{N-1}{A}-\frac{\varepsilon}{L}}-A\right).
$$
\begin{figure}[H]
	\begin{center}
            \includegraphics[height=8.0cm]{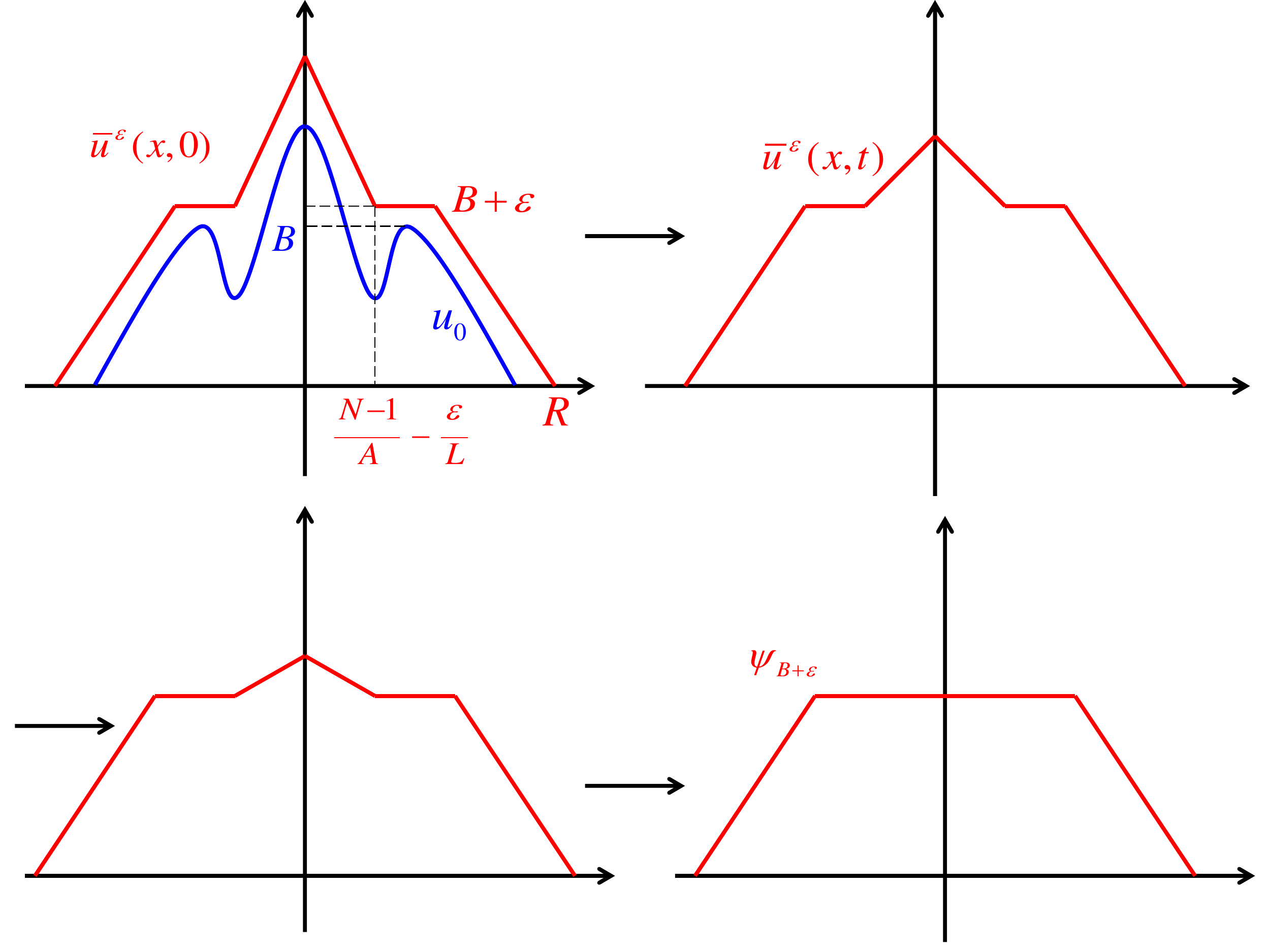}
		\vskip 0pt
		\caption{Supersolution $\overline{u}^{\varepsilon}$}
        \label{fig:super2}
	\end{center}
\end{figure}
 Recall $B=\max\limits_{(N-1)/A\leq |x|\leq R}u_0(x)\leq \lambda(R-(N-1)/A)$, and $u_0$ is $L$-Lipschitz continuous. It is easy to check 
$$
\psi^-\leq \overline{u}^{\varepsilon}\leq \psi^+\quad \text{and}\quad \overline{u}^{\varepsilon}(\cdot,0)\geq u_0.
$$
For $|x_0|>\frac{N-1}{A}-\frac{\varepsilon}{L}$, and $t_0>0$, obviously $\overline{u}^{\varepsilon}$ is supersolution in the viscosity sense.
\\
For $|x_0|<\frac{N-1}{A}-\frac{\varepsilon}{L}$, and $t_0>0$,
\begin{align*}
&\left(\overline{u}^{\varepsilon}_{t}-|D \overline{u}^{\varepsilon}|{\rm div}\left(\frac{D \overline{u}^{\varepsilon}}{|D \overline{u}^{\varepsilon}|}\right)-A|D \overline{u}^{\varepsilon}|\right)(x_0,t_0)\\
= \ &L e^{-\nu_2 t_0} \left(-\nu_2(\frac{N-1}{A}-\frac{\varepsilon}{L}-|x_0|)+\frac{N-1}{|x_0|}-A\right)\\
>\ &L e^{-\nu_2 t_0} \left(-\nu_2(\frac{N-1}{A}-\frac{\varepsilon}{L})+\frac{N-1}{\frac{N-1}{A}-\frac{\varepsilon}{L}}-A\right)>0.
\end{align*}
For $|x_0|=\frac{N-1}{A}-\frac{\varepsilon}{L}$, $t_0>0$, and test function $\varphi\in C^2(\mathbb{R}^N\times(0,\infty))$ satisfies  
$$
(\varphi-\overline{u}^{\varepsilon})(x_0,t_0)<(\varphi-\overline{u}^{\varepsilon})(x,t),\quad (x,t)\neq(x_0,t_0).
$$
Then 
$$
\varphi_t(x_0,t_0)\geq0.
$$
By abuse of notions, we write $\overline{u}^{\varepsilon}(x,t)=\overline{u}^{\varepsilon}(r,t)$ for $r=|x|$. It is easy to see 
$$
\lim\limits_{r\rightarrow(\frac{N-1}{A}-\frac{\varepsilon}{L})_-}\overline{u}^{\varepsilon}_r(r,t)= -Le^{-\nu_2t},
$$
and
$$
\lim\limits_{r\rightarrow(\frac{N-1}{A}-\frac{\varepsilon}{L})_+}\overline{u}^{\varepsilon}_r(r,t)=0\quad \text{or}\quad -\lambda.
$$
Then by Lemma \ref{lem:sub}, there exists $s\leq0$ such that
$$
D \varphi(x_0,t_0)=\frac{sx_0}{|x_0|}\quad \text{and}\quad {\rm tr}\left[\left((I-\frac{Du\otimes Du}{|Du|^2})\right)D^2 u \right](x_0,t_0)\leq \frac{(N-1)s}{|x_0|}.
$$
Since $|x_0|=\frac{N-1}{A}-\frac{\varepsilon}{L}< (N-1)/A$, and $s\leq 0$, we have
$$
\left(\varphi_t-|D \varphi|{\rm div}\left(\frac{D \varphi}{|D \varphi|}\right)-A|D \varphi|\right)(x_0,t_0)\geq -\frac{(N-1)s}{|x_0|}-A|s|=-\frac{(N-1)s}{|x_0|}+As> 0.
$$
Thus, $\overline{u}^{\varepsilon}$ is a supersolution of (\ref{eq:levelset}), (\ref{eq:initial}), and (\ref{eq:obs}).

Next, we construct a subsolution $\underline{u}=\underline{\varphi}\vee\psi^-$ to equation (\ref{eq:levelset}), (\ref{eq:initial}), and (\ref{eq:obs}), where
$$
\underline{\varphi}(x,t)=\left\{
\begin{array}{lcl}
-Le^{-\nu_3t}(r_0-|x|)+B,\quad |x|\leq r_0,\\
(L(r_0-|x|)+B)\vee (Le^{-\nu_4,t}(|x|-R)),\quad r_0<|x|\leq R,\\
0,\quad |x|>R.
\end{array}
\right.
$$
\begin{figure}[htbp]
	\begin{center}
            \includegraphics[height=8.0cm]{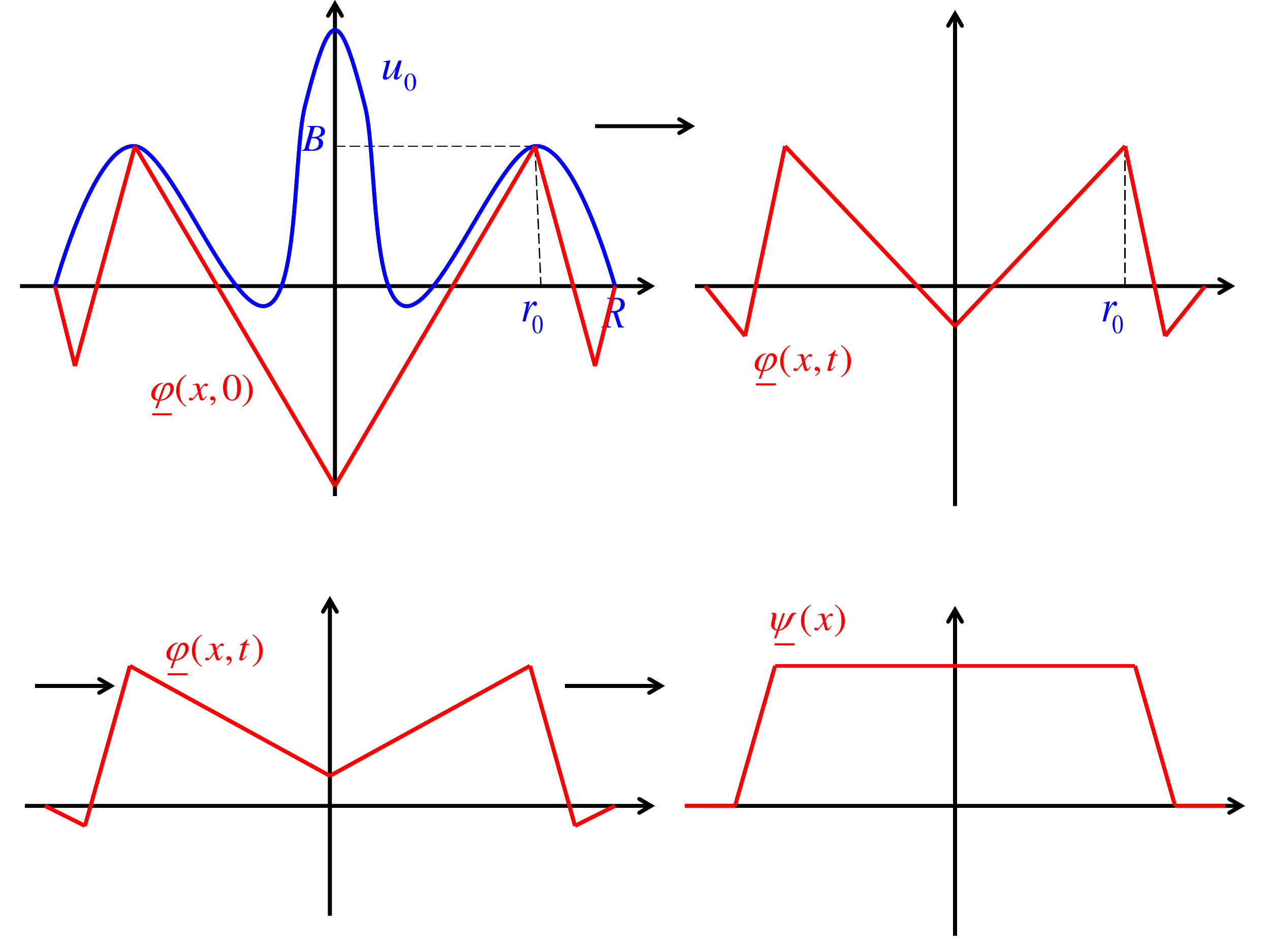}
		\vskip 0pt
		\caption{Subsolution $\underline{\varphi}$}
        \label{fig:sub2}
	\end{center}
\end{figure}
Here $r_0\in[(N-1)/A, R]$ satisfies $B=u_0(r_0)$,
$$
0<\nu_3<\frac{N-1+Ar_0}{r_0^2},
$$
and
$$
0<\nu_4<\frac{1}{R}\left(\frac{N-1}{R}+A\right).
$$
It is easy to check $u_0\geq \underline{u}(\cdot,0)$, and 
$$
\psi^-\leq \underline{u}\leq\psi^+.
$$
Clearly, $\underline{u}$ is a subsolution for $|x|\geq R$, $t>0$. We show the claim for $|x|<R$, $t>0$.
For $r_0<|x_0|<R$, $t_0>0$, if $\underline{\varphi}=Le^{-\nu_4,t}(|x|-R)$, we compute
\begin{multline*}
\left(\underline{\varphi}_t-|D \underline{\varphi}|{\rm div}\left(\frac{D \underline{\varphi}}{|D \underline{\varphi}|}\right)-A|D \underline{\varphi}|\right)(x_0,t_0)=L e^{-\nu_4 t_0} \left(\nu_4(R-|x_0|)-\frac{N-1}{|x_0|}-A\right)\\
\dis{\leq L e^{-\nu_4 t_0} \left(\nu_4 R-\frac{N-1}{R}-A\right)<0.}
\end{multline*}
For $r_0<|x_0|< R$, $t_0>0$, if $\underline{\varphi}=L(r_0-|x|)+B$, we compute 
\[
\left(\underline{\varphi}_{t}-|D \underline{\varphi}|{\rm div}\left(\frac{D \underline{\varphi}}{|D \underline{\varphi}|}\right)-A|D \underline{\varphi}|\right)(x_0,t_0)=L \left(\frac{N-1}{|x_0|}-A\right)
\dis{<L(\frac{N-1}{r_0}-A)\leq 0}.
\]
For $|x_0|<r_0$, $t_0>0$,
\begin{multline*}
\left(\underline{\varphi}_{t}-|D \underline{\varphi}|{\rm div}\left(\frac{D \underline{\varphi}}{|D \underline{\varphi}|}\right)-A|D \underline{\varphi}|\right)(x_0,t_0)=L e^{-\nu_3 t_0} \left(\nu_3(r_0-|x_0|)-\frac{N-1}{|x_0|}-A\right)\\
<L e^{-\nu_3 t_0} \left(\nu_3r_0-\frac{N-1}{r_0}-A\right)<0.
\end{multline*}
Finally, we need to check the case where $|x_0|=r_0$, $t_0>0$. Take a test function $\varphi\in C^2(\mathbb{R}^N\times(0,\infty))$ such that $\underline{\varphi}-\varphi$ has a strict maximum at $(x_0,t_0)$.
Then
$$
\varphi_t(x_0,t_0)\leq 0.
$$
By Lemma \ref{lem:sub}, there exists $s\in[-L,Le^{-\nu_3t_0}]$ such that
$$
D \varphi(x_0,t_0)=\frac{sx_0}{|x_0|}\quad \text{and}\quad {\rm tr}\left[\left((I-\frac{Du\otimes Du}{|Du|^2})\right)D^2 u \right]\geq \frac{(N-1)s}{|x_0|}.
$$
We use the above and the fact that $|x_0|=r_0\geq (N-1)/A$ to imply
$$
\left(\varphi_t-|D \varphi|{\rm div}\left(\frac{D \varphi}{|D \varphi|}\right)-A|D \varphi|\right)(x_0,t_0)\leq -\frac{(N-1)s}{|x_0|}-A|s|\leq
\frac{(N-1)|s|}{|x_0|}-A|s|\leq 0,
$$
 Therefore, $\underline{u}$ is a subsolution.
By the comparison principle,
$$
\underline{u}\leq u\leq \overline{u}^{\varepsilon}.
$$
Letting $t\rightarrow\infty$ and $\varepsilon\rightarrow0$ in this order,
$$
\underline{\psi}\leq v\leq B\wedge\psi^+=\psi_B.
$$
Here 
$$
\underline{\psi}(x)=\left\{
\begin{array}{lcl}
B,\quad |x|\leq r_0,\\
L(r_0-|x|)+B,\quad r_0<|x|\leq r_0+\frac{B}{L},\\
0,\quad r_0+\frac{B}{L}<|x|. 
\end{array}
\right.
$$
Proposition \ref{prop:symmetricstation} shows that $v=\psi_B$.
\begin{figure}[htbp]
	\begin{center}
            \includegraphics[height=8.0cm]{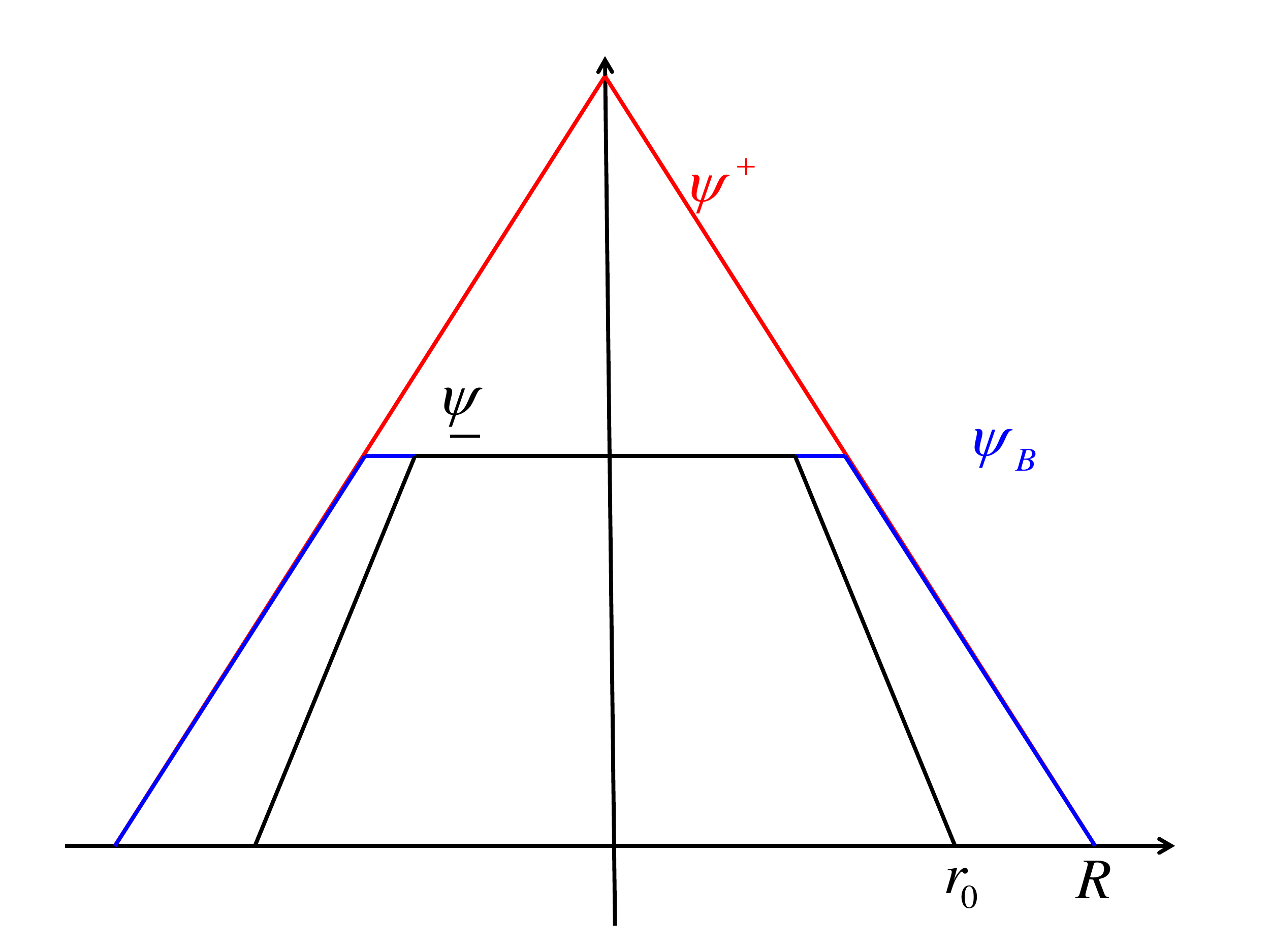}
		\vskip 0pt
		\caption{$\underline{\psi}$ and $\psi_B$}
        \label{fig:sta2}
	\end{center}
\end{figure}
\end{proof}

\begin{rem}
Theorem \ref{thm:asymcircle} also holds for $\psi^-\leq 0$ in $\mathbb{R}^N$ by using approximation arguments. We leave the proof to the readers.
\end{rem}

\begin{rem}
We give another idea to prove Theorem \ref{thm:asymcircle} as following. Recall 
$$
u(x,t)=u(|x|,t)=u(r,t).
$$
For $r>(N-1)/A$ such that $\psi^-(r)<u(r,t)<\psi^+(r)$, we have the following equation 
$$
u_t(r,t)=\frac{N-1}{r}u_r+A|u_r|\geq0.
$$
Therefore, at $r_0$, we yield that $t\mapsto u(r_0,t)$ is non-decreasing. This immediately gives us the desired conclusion for $r>(N-1)/A$.
\end{rem}

\section{Appendix}
\begin{lem}\label{lem:subproof} Let $N=2$. The function $\underline{\psi}$ is a subsolution of problem
\begin{equation}\label{eq:app}
\left\{
\begin{array}{lcl}
u_t-|D u|{\rm div}\dis{\left(\frac{D u}{|D u|}\right)}-|D u|=0,\quad (x,t)\in B_2(O)\times(0,\infty),\\
u(x,0)=u_0(x),\quad x\in \overline{B_2(O)},\\
u(x,t)=0,\quad (x,t)\in\partial B_2(O)\times[0,\infty),
\end{array}
\right.
\end{equation}
and satisfy
$$
\left\Vert\underline{\psi}\right\Vert_{C^{0,1}\left(\,\overline{B_2(O)}\times[0,\infty)\right)}=\infty,
$$
where $u_0\geq\underline{\psi}(\cdot,0)$ in $ \overline{B_2(O)}$, and
$$
\underline{\psi}(x,t)=\left\{
\begin{array}{lcl}
1,\quad |x|\leq 2-\frac{1}{2}e^{-\frac{1}{6}t},\\
2e^{\frac{1}{6}t}(2-|x|),\quad 2-\frac{1}{2}e^{-\frac{1}{6}t}<|x|\leq 2.\\
\end{array}
\right.
$$ 
\end{lem}
\begin{proof} Obviously, $\psi(\cdot,t)=0$ on $\partial B_2(O)$. There is nothing to check if $|x_0|<2-\frac{1}{2}e^{-\frac{1}{6}t_0}$, and $t_0>0$. 

If $2-\frac{1}{2}e^{-\frac{1}{6}t_0}<|x_0|<2$, and $t_0>0$, noting $|x_0|>\frac{3}{2}$, we compute
\begin{multline*}
\left(\underline{\psi}_t-|D \underline{\psi}|{\rm div}\left(\frac{D \underline{\psi}}{|D \underline{\psi}|}\right)-|D \underline{\psi}|\right)(x_0,t_0)=2 e^{\frac{1}{6} t_0} \left(\frac{1}{6}(2-|x_0|)+\frac{1}{|x_0|}-1\right)\\
\dis{\leq 2 e^{\frac{1}{6} t_0} \left(\frac{1}{6}(2-|x_0|)+\frac{1}{|x_0|}-1\right)}\leq\dis{ 2 e^{\frac{1}{6} t_0} \left(\frac{1}{12}+\frac{2}{3}-1\right)}<0.
\end{multline*}
Finally, we check the case where $|x_0|=2-\frac{1}{2}e^{-\frac{1}{6}t_0}$, and $t_0>0$. Choose a test function $\varphi\in C^2(\mathbb{R}^2\times(0,\infty))$ such that $\underline{\psi}-\varphi$ has a strict maximum at $(x_0,t_0)$.
By Lemma \ref{lem:sub}, there exists $s\in[-\frac{1}{2}e^{\frac{1}{6}t_0},0]$ such that
$$
D \varphi(x_0,t_0)=\frac{sx_0}{|x_0|}\quad \text{and}\quad {\rm tr}\left[\left((I-\frac{Du\otimes Du}{|Du|^2})\right)D^2 u \right]\geq \frac{s}{|x_0|}.
$$
Let $x(t)=(2-\frac{1}{2}e^{-\frac{1}{6}t})\frac{x_0}{|x_0|}$. For $t<t_0$, $1=\underline{\psi}(x(t),t)<\varphi(x(t),t)$, and $1=\varphi(x(t_0),t_0)$. Therefore, 
$$
0\geq\frac{\rm d}{{\rm d} t}\varphi(x(t),t)\mid_{t=t_0}=\frac{1}{12}e^{-\frac{1}{6}t_0}\frac{x_0}{|x_0|}\cdot D\varphi(x_0,t_0)+\varphi_t(x_0,t_0)=\frac{1}{12}se^{-\frac{1}{6}t_0}+\varphi_t(x_0,t_0).
$$
Thus, for $s\leq0$, 
\begin{multline*}
\left(\varphi_t-|D \varphi|{\rm div}\left(\frac{D \varphi}{|D \varphi|}\right)-|D \varphi|\right)(x_0,t_0)\leq -\frac{1}{12}se^{-\frac{1}{6}t_0}-\frac{s}{|x_0|}-|s|\\
\dis{\leq-\frac{s}{12}e^{-\frac{1}{6}t_0}-\frac{s}{2}+s\leq-\frac{s}{12}-\frac{s}{2}+s}<0.
\end{multline*}
Hence, $\underline{\psi}$ is a subsolution.

Next we compute
$$
\frac{\partial}{\partial {\bf n}} \underline{\psi}(x,t),\quad \text{for} \ x\in\partial B_2(O),
$$
where ${\bf n}$ denotes the outward normal vector on $\partial B_2(O)$.
For $x_0\in \partial B_2(O)$, $t>0$,
$$
\frac{\partial}{\partial {\bf n}} \underline{\psi}(x_0,t)=-2e^{\frac{1}{6}t}.
$$
Therefore,
$$
\lim\limits_{t\rightarrow\infty}\left|\frac{\partial}{\partial {\bf n}} \underline{\psi}(x_0,t)\right|=\infty.
$$
We complete the proof.
\end{proof}
\begin{cor}
Assume $u$ is a solution of problem {\rm (\ref{eq:app})} satisfies $u=0$ on the boundary in the 
classical sense. If $u_0\geq \underline{\psi}(\cdot,0)$ on $\partial B_2(O)$, then 
$$
\left\Vert u\right\Vert_{C^{0,1}\left(\,\overline{B_2(O)}\times[0,\infty)\right)}=\infty.
$$
\end{cor}

\section{Open problems}

In this paper, we get precise behaviors of the limiting profile for special obstacles, and initial data in Theorem \ref{thm:asymcircle}. The following problems are open, and of great interests.

\noindent {\bf Problem A.} Assume that $\psi^\pm$, and $u_0$ are radially symmetric (but not of the precise forms of Theorem \ref{thm:asymcircle}). Study the limiting profile and its dependence on $\psi^\pm$, and $u_0$.

\noindent {\bf Problem B.} Let $\psi^\pm$ be as in Theorem \ref{thm:asymcircle}. Study the limiting profile and its dependence on $\psi^\pm$, and $u_0$.

It is important noting that we do not assume $u_0$ is radially symmetric in Problem B, and therefore, $u$ needs not be radially symmetric. In order to understand clearly the behavior of $v(x)=\lim_{t \to \infty} u(x,t)$, we need to characterize all stationary solutions of (\ref{eq:levelset}$e$), and (\ref{eq:obs}), which is not yet known in the literature.

Finally, a most general, and most challenging problem is as following.

\noindent {\bf Problem C.} Characterize all stationary solutions of (\ref{eq:levelset}$e$), and (\ref{eq:obs}) in the general setting. Use this to study the limiting profile.

\end{document}